\documentclass[reqno,14pt]{amsart}

\usepackage[left=2.5cm, top=3cm, right=2.5cm, bottom=3cm]{geometry}
\usepackage[utf8]{inputenc}
\usepackage{amssymb,bbm}
\usepackage{amsfonts}
\usepackage{amsmath, amscd,amsthm}
\usepackage{graphicx}
\usepackage{color}
\usepackage{enumerate}
\usepackage{stmaryrd}
\usepackage{bm}
\usepackage{mathtools}
\usepackage{fancyvrb}
\usepackage{soul}

\def \N{\mathbb{N}}

\def \Z{\mathbb{Z}}
\def \V{\mathbb{V}}

\def \P{\mathbb{P}}
\newcommand{\E}{\mathbb{E}}
\newcommand{\Var}{\V\mathrm{ar}}

\newtheorem{lemma}{Lemma}

\newtheorem{proposition}{Proposition}
\newtheorem{theorem}{Theorem}

\newtheorem*{remark}{Remark}

\date{}


\title{Estimates on Escape Times for the Elephant Random Walk}

\author{Morgan Andr\'e}
\address{Morgan Andr\'e
		\newline
		Universidade de S\~ao Paulo,
		\newline  R. do Mat\~ao, 1010 - Butant\~a, S\~ao Paulo - SP, CEP: 05508-090, Brasil.
		\newline
		e-mail: {\rm \texttt{morgan@ime.usp.br }}
        }

\author{Leonel Zuazn\'abar }
\address{Leonel Zuazn\'abar
		\newline
		Universidade Federal do ABC.,
		\newline Av. dos Estados, 5001 - Bangú, Santo André - SP, CEP: 09280-560, Brasil.
		\newline
		e-mail: {\rm \texttt{l.zuaznabar@ufabc.edu.br}}
        }

\begin{document}

\maketitle

\begin{abstract}
We study the gambler's ruin problem for the Elephant Random Walk, focusing on escape time from a symmetric interval of the form $\{-N, \ldots, N\}$. As our main result, we derive tight exponential bounds for the tail of this escape time. We then illustrate the usefulness of such bounds by proving that, in the diffusive regime, the Elephant's average behavior mirrors that of the traditional symmetric random walk: the expected escape time grows quadratically with $N$.
\end{abstract}

\vspace{2 cm}

\begin{center}
\centering

\end{center}

\normalsize

\section{Introduction}

The gambler’s ruin may be the oldest problem in probability theory, dating back to a letter from Blaise Pascal to Pierre de Fermat in 1656. In the classical gambler's ruin, a gambler is betting on the outcome of a possibly biased coin, winning or loosing one dollar at each step, with probability $p$ and $1-p$ respectively, independently of what happened up to the current step. Assume that the gambler starts with fortune $N$ and stop either when the totality of its fortune has been lost, or when it has doubled. For simplicity we adopt the representation in which $S_n$ represents the gambler's gains (when $S_n \geq 0$) or losses (when $S_n \leq 0$) with $S_0=0$. Thus $(S_n)_{n \ge 0}$ is a Simple Random Walk (SRW) on the one-dimensional integer lattice, starting at the origin. Moreover we denote the game's duration by $\sigma_N = \inf\left\{n\ge 1: \left|S_n\right| = N \right\}$. One of the main object of interest is the expected duration of the game, and, in this standard (Markovian and time-homogeneous) setting the problem has long been fully and analytically solved. Classical results\footnote{See for example \cite{feller}, pages 348 and 349.} tell us that $\E[\sigma_N]$ is of quadratic order in the symmetric case, while it is of linear order in the asymmetric case. More precisely, if $p=1/2$, then $\E \left[\sigma_N \right] = N^2$, otherwise: $$\E \left[\sigma_N \right] = \frac{N}{2p - 1} \left( \frac{1 - \left( \frac{1-p}{p}\right)^N}{1 + \left( \frac{1-p}{p}\right)^N}\right).$$
The proofs of these closed-form formulas are highly dependent on the Markovian and time-homogeneous nature of the dynamic. In the present work, we develop an approach allowing the analysis of a case in which the random walk has a past-dependent dynamic. Namely, we consider the problem above when the classical simple random walk is replaced by the elephant's random walk.

\vspace{0.2 cm}

The Elephant's Random Walk (ERW) is a discrete time stochastic process introduced by Schütz and Trimper \cite{schutz2004elephants} as a paradigmatic example of past-dependent random walks. It is both quantitatively tractable and qualitatively rich---exhibiting in particular a phase transition from diffusive to super-diffusive behavior---and has therefore recently become the standard choice for testing classical results outside the traditional Markovian setting, when memory is allowed to play a role. It can be defined as follows. Let $(Z_n)_{n \geq 0}$ be a stochastic chain with value in $\Z$, making jumps of unit size---either upward or downward---at each time step. It can be defined, like classical random walks, as the partial sum of a sequence $(X_n)_{n \geq 0}$ of random variables: $$Z_n = \sum_{k=0}^n X_k,$$ with the crucial difference that the $(X_n)_{n \geq 0}$ are not independent. Conventionally $X_0 = 0$, and the first jump leads to either $X_1 = 1$ or $X_1 = -1$, with probabilities that are unimportant for our purposes, since in either cases $|Z_1|=1$. Then, for any $n \geq 1$, the jump at time $n+1$ is given by:
$$X_{n+1} = \begin{cases} X_{K(n)} &\text{ with probability } p,\\
-X_{K(n)} &\text{ with probability } 1-p,
\end{cases}$$
where $(K(n))_{n \geq 1}$ is a sequence of independent random variables, $K(n)$ having uniform distribution on $\{1, \ldots, n\}$. In words, at each time step, the ERW selects a step uniformly at random from its entire past and repeats it with probability $p$, or takes the opposite step with probability $1 - p$. The parameter $p$ is often referred to as the \textit{memory} of the elephant: when $p$ is close to $1$, the dependence on the past tends to persist over a long period, whereas it fades rapidly when $p$ is close to $0$.

\vspace{0.2 cm}

Much effort has been done during the last two decades to understand the effect of the memory of the elephant on its asymptotic behavior, as compared to classical Markovian random walks, with a particular focus on laws of large numbers, central limit theorems and functional limit theorems---see for example \cite{baur2016elephant, bercu, coletticlt, coletticlt2}. Results on hitting times have been obtained, but mostly from the view point of \textit{return times} and recurrence/transience---see \cite{bertoinZeros, colettiRec}. Surprisingly, such natural questions as the gambler's ruin problem---in other words, the question of \textit{escape times}---has, to the best of our knowledge, never been addressed for the ERW, even though it was briefly mentioned in \cite{stadtmuller} and \cite{pozdnyakov}.

\vspace{0.2 cm}

The question is whether or not the classical behavior mentioned at the beginning of this section remains when $(S_n)_{n \ge 0}$ is replaced by $(Z_n)_{n \ge 0}$. For concreteness, one can imagine a gambling game in which the coin is replaced by Pólya's urn\footnote{See \cite{baur2016elephant} for details on the connection between ERW and Pólya's urn.}. More precisely, suppose an urn contains initially a unique ball, which is either green or red. Then at any time $n \geq 1$ a ball is drawn at random (uniformly) from the urn, and then put back in with another ball, which has either the same color or the other color, with probability $p$ and $1-p$ respectively. Then replacing $(S_n)_{n \geq 0}$ by $(Z_n)_{n \geq 0}$ is the same as saying that the gambler is betting on the outcome of the urn sampling, and win or loose one dollar at time $n$ depending on whether the ball drawn at time $n$ is green or not. Define:

\begin{equation*}
    \tau_N = \inf\left\{k\ge 1: |Z_k| = N\right\}.
\end{equation*}

\vspace{0.2 cm}

It is well-known (see \cite{schutz2004elephants}) that the behavior of the ERW critically depends on the value of $p$, with a phase transition at $3/4$. $(Z_n)_{n \geq 0}$ is said to be in the \textit{diffusive regime} when $p < 3/4$, and in the \textit{super-diffusive} regime when $p > 3/4$. We prove that, in the diffusive regime, the expectation of the game duration exhibits similar quadratic behavior as in the classical symmetric gambler's ruin for big $N$. More precisely, we prove the following.

\begin{proposition}\label{prop_1}
    If $0<p<3/4$ then there exists a constant $\theta$ such that: $$\E[\tau_N] \underset{N \rightarrow\infty}{\sim} \theta N^2. $$
\end{proposition}

\vspace{0.2 cm}

The above result is obtained as a corollary of a much stronger one, which might actually be considered as the main contribution of this work. Namely we obtain tight bounds on the tail of the game's duration of the elephant's ruin.  

\begin{theorem}\label{main_thm} If $0<p<3/4$ then there exists positive constants $C_1$, $C_2$, $C_3$ and $C_4$ such that, for all $N \ge 2$ and $t \geq 1$, the following bound holds:

    \begin{equation*}
        C_1 \exp \left( -C_2 \frac{t}{N^2} \right) \leq \P \left( \tau_N > t \right) \le C_3 \exp \left(-C_4\frac{t}{N^2}\right).
    \end{equation*}
\end{theorem}

\begin{remark}
    Even though formally the result above is stated only for $0 < p < 3/4$, the upper bound actually holds for all $0 < p < 1$. Moreover, some of the constants above may depend on $p$.
\end{remark}

The main idea behind the proof of Theorem \ref{main_thm} is to first establish analogous bounds for the SRW, and then demonstrate via couplings that, on a quadratic time scale, the ERW does not deviate excessively from the SRW. These couplings rely on a well-known time-inhomogeneous Markovian description of the ERW. Given a sequence of i.i.d. random variables with uniform distribution on $(0,1)$, we can construct both the ERW and the SRW on the same probability space. Each walk determines its next jump based on the value of the current uniform random variable relative to a specified threshold. For the SRW this threshold is always $1/2$, while for the ERW it depends on both time and the walk’s current state—though it is always either greater or less than $1/2$, depending on whether $p>1/2$ or not. This construction yields joint realizations in which the ERW either dominates or is dominated by the SRW, according to the value of $p$. By leveraging the SRW's bounds, this approach already resolves half of the cases. To address the remaining cases and thus complete the proof of Theorem \ref{main_thm}, one then needs to control the distance between the SRW and the ERW on the appropriate time scale. Specifically, for $t = cN^2$, we identify some scaling constant $A$ such that, whenever the ERW or the SRW reaches the boundaries of $[-AN, AN]$, the other one has vanishingly small probability of not having reached $\{-N,N\}$ yet. This allows us prove Theorem \ref{main_thm} on the quadratic time scale $t = cN^2$. Then, by partitioning the interval $[0, t]$ into sub-intervals of length $t/cN^2$, we extend the result to any $t \geq 1$. Finally, Proposition \ref{prop_1} follows directly from the functional limit theorem in \cite{baur2016elephant} and uniform integrability.

\vspace{0.2 cm}

The paper is organized as follows. In Section 2 we derive a well-known (time inhomogeneous) Markovian description of the ERW allowing a coupling with the simple random walks. In Section 3 we prove the exponential bounds on the tail of the simple random walk. In Section 4 we define a coupling between ERW and the classical random walks and prove our main result, Theorem \ref{main_thm}. Finally, in Section 5, we prove Proposition \ref{prop_1}.

\vspace{0.2 cm}

\section{Markovian description of the ERW}

Even though the ERW appears as highly non-Markovian, it is actually a well-established fact that its dynamic can be described in a Markovian fashion, with the caveat that the transition kernel is time-inhomogeneous. Writing $L_n$ (resp. $R_n$) for the total number of $-1$ steps (resp. $+1$ steps) at time $n$, the following holds $$\begin{cases}
L_n + R_n = n\\
R_n - L_n = Z_n
\end{cases} \Leftrightarrow \; \; \; \;
\begin{cases}
R_n = \frac{n + Z_n}{2}\\
L_n = \frac{n-Z_n}{2}
\end{cases} 
$$
At that time, the ERW goes right with probability $p \cdot \frac{R_n}{n} + (1-p) \cdot \frac{L_n}{n} = \frac{1}{2} + \frac{(2p-1) Z_n}{2n}$,
and left with complementary probability. In other words  $(Z_n)_{n \geq 0}$ is a time-inhomogeneous Markov chain with transition probabilities at time $n \geq 1$ given by:
$$ \P \left( Z_{n+1} = j \, \middle| \, Z_n = i \right) =
\begin{cases}
 \frac{1}{2} + \frac{(2p-1) i}{2n} &\text{ if } j = i+1,\\
 \frac{1}{2} - \frac{(2p-1) i}{2n} &\text{ if } j = i-1,\\
 0 &\text{otherwise.}
 \end{cases}
$$

\vspace{0.2 cm}

As can be seen above, the ERW transition kernel presents a (time and space dependent) drift, oriented toward the origin when $p<1/2$, and oriented contrariwise when $p>1/2$. When $p = 1/2$ the ERW is simply the classical simple random walk on $\mathbbm{Z}$. Our purpose being the study of the ERW on $\Z \cap \left[-N, \ldots, N\right]$, we might as well study the absolute value of $(Z_n)_{n \geq 0}$, and consider the time it hits $N$. Taking advantage of the symmetry one can easily check that the transition kernel remains almost unchanged:
\begin{align*}
\P \left( |Z_{n+1}| = j \, \middle| \, |Z_n| = i\right) = \begin{cases}
 1 &\text{ if } i=0 \text{ and }  j = 1,\\
 \frac{1}{2} + \frac{(2p-1) i}{2n} &\text{ if } i > 0 \text{ and } j = i+1,\\
 \frac{1}{2} - \frac{(2p-1) i}{2n} &\text{ if } i > 0 \text{ and } j = i-1,\\
 0 &\text{otherwise.}
 \end{cases}
\end{align*}

\vspace{0.2 cm}

\section{The Elephant's tail}

In this section we prove Theorem \ref{main_thm}. This is done by mean of a coupling between the ERW and the SRW, leveraging analogous results for the SRW, stated in the following lemma and proven in the appendix.

\begin{lemma}\label{lemma:sym} There exists positive absolute constants $C_1$, $C_2$, $C_3$ and $C_4$ such that 
\begin{equation*}
    C_1 \exp \left(-C_2\frac{t}{N^2}\right) \le \mathbb{P}(\sigma_N  > t) \le C_3 \exp \left(-C_4\frac{t}{N^2}\right),
\end{equation*}
for all $N\ge 2$ and $t \ge  1$.  
\end{lemma}

Now let $(U_k)_{k\ge 1}$ be a sequence of i.i.d. random variables with uniform distribution on $(0,1)$. Let $\tilde{Z}_0 = 0$ and, for some $p \in (0,1)$, define recursively $\eta_k$ and $\tilde{Z}_k$ as follows:
\begin{equation*}
\eta_{k} = \mathbbm{1}_{\left\{\tilde{Z}_{k-1} = 0 \right\}} +  \mathbbm{1}_{\left\{\tilde{Z}_{k-1} \neq 0\right\}} \left[ \mathbbm{1}_{\left\{U_{k} < \frac{1}{2} + \frac{(2p - 1)}{k- 1}\Tilde{Z}_{k-1} \right\}} - \mathbbm{1}_{\left\{U_{k} \ge  \frac{1}{2} + \frac{(2p - 1)}{k - 1}\Tilde{Z}_{k-1} \right\}}\right],
\end{equation*}
and
\begin{equation*}
\tilde{Z}_{k} = \sum_{j=1}^{k}\eta_j.
\end{equation*}
Observe that from the precedent section it follows that $(\tilde{Z}_{k})_{k\ge 0}$ and $(|Z_k|)_{k\ge 0}$ have the same distribution.

\subsection{Proof of the upper bound in Theorem \ref{main_thm}}

    Below we break the proof into two separate cases. First we consider the easy one in which the elephant is drifted to the right ($p \geq 1/2)$. In that case the ERW and the SRW can be coupled in such a way that the SRW is dominated by the ERW, so that the result follows straightforwardly from Lemma \ref{lemma:sym}. Then we consider the harder case in which the elephant is drifted to left ($p<1/2$). The difficulty arises from the fact that, in this case, the domination is reversed. The main observation is that, in the diffusive regime, one expects the drift of the ERW to vanish as times goes by, so that if one can control how fast it does so, the SRW bounds from Lemma \ref{lemma:sym} might still be of some use. In order to do so we let the ERW wander around for some time, say $dN^2$. If it has not reached the barrier $\{-N,N\}$ yet at that point, then the drift shall be sufficiently small for the ERW to start behaving essentially like the SRW. Thus, with high probability, the additional time it would take for a SRW to go as far as $\{-2N,2N\}$, while starting at the same position as the ERW at time $dN^2$, shall be sufficient for the ERW to reach $\{-N,N\}$ in the meantime. Then, for some constant $c$, Lemma \ref{lemma:sym} gives us a uniform upper bound for $\mathbbm{P}_x(\tau_{N} > cN^2)$, which can easily be turned into an exponential upper-bound via the strong Markov property. 

    \medskip

    \begin{enumerate}[(i)]
    \item \textbf{Case $\bm{p \ge 1/2}$.}
    Let $\tilde{S}_0 = 0$ and, for $k\ge 1$, define recursively $\xi_k$ and $\tilde{S}_k$ by:
    \begin{equation*}
        \xi_{k} = \mathbbm{1}_{\left\{\tilde{S}_{k-1} = 0 \right\}} +  \mathbbm{1}_{\left\{\tilde{S}_{k-1} \neq 0 \right\}} \left( \mathbbm{1}_{\left\{U_{k} < \frac{1}{2} \right\}} - \mathbbm{1}_{\left\{U_{k} \ge  \frac{1}{2}\right\}}\right), 
    \end{equation*}
    and
    \begin{equation}\label{srw-coupling}
        \tilde{S}_{k} = \sum_{j=1}^{k}\xi_j.
    \end{equation}
    Then $\tilde{S}_k$ has the same distribution as $|S_k|$, where $(S_k)_{k \ge 0}$ is still the symmetric simple random walk defined in the previous section. Moreover, $(\tilde{S}_k)_{k\ge 0}$ and $(\tilde{Z}_{k} )_{k\ge 0}$ are constructed in such a way that the following holds:
    \begin{equation}\label{enq_1}
        \tilde{Z}_k \ge \tilde{S}_k, \, \text{ for all } k\ge 0.
    \end{equation}
    Notice that the only way inequality \eqref{enq_1} could be broken would be if at some point $\tilde{Z}_k = 1$ and $\tilde{S}_k = 0$ (and $\tilde{Z}_{k}$ jumps to the left between $k$ and $k+1$). Fortunately this never happens since $\tilde{Z}_k$ and $\tilde{S}_k$ have the same parity. Below, we deliberately abuse notation, redefining $\tau_N$ and $\sigma_N$ in terms of $(\tilde{S}_k)_{k\ge 0}$ and $(\tilde{Z}_{k} )_{k\ge 0}$. By inequality \eqref{enq_1} and Lemma \ref{lemma:sym}, for all $N\ge 1$ and $t\ge 1$, one indeed has:
    
    \begin{equation*}
        \P(\tau_N > t) \le \P(\sigma_N > t) \le C_3e^{-C_4t}.
    \end{equation*}
   
    \medskip
   
    \item \textbf{Case $\bm{p < 1/2}$.} Let $d>0$ be some positive constant, which exact value will be chosen later, and let the definition of $(\tilde{S}_k)_{k\ge 0}$ be slightly modified. $\tilde{S}_0 = \tilde{Z}_{dN^{2}}$ and, for $k \ge 1$, define recursively $\zeta_k$ and $\tilde{S}_k$ as follows:
    
    \begin{equation*}
        \zeta_{k} = \mathbbm{1}_{\left\{\tilde{S}_{k-1} = 0 \right\}} + \mathbbm{1}_{\left\{\tilde{S}_{k-1}\neq 0\right\}}\left[\mathbbm{1}_{\left\{U_{k + dN^{2}} < \frac{1}{2}\right\}} - \mathbbm{1}_{\left\{U_{k + dN^{2}} \ge  \frac{1}{2} \right\} }\right],
    \end{equation*}
    and
    \begin{equation*}
        \tilde{S}_k = \tilde{S}_0 + \sum_{j=1}^k \zeta_j.
    \end{equation*}

    In words, $(\tilde{S}_k)_{k\ge 0}$ is still distributed as $(|S_k|)_{k \ge 0}$, but starting at a random position depending on the trajectory of $(\tilde{Z}_k)_{k\ge 0}$ up to time $dN^2$. Furthermore, let $\sigma_N$ be momentarily redefined once again in term of this new $(\tilde{S}_k)_{k\ge 0}$. The coupling is now shifted to the right with respect to time, and---since $p <1/2$---it goes the other way around:
    \begin{equation}
        \tilde{Z}_{d N^2 + k} \le \tilde{S}_k, \, \text{ for all } k\ge 0.
    \end{equation}
    We can nonetheless control the distance between the two. Indeed, between times $j-1$ and $j$, the distance increases by two units if and only if: $$U_j \in \left[\frac{1}{2} + \frac{(2p-1)\tilde{Z}_{j +dN^{2}}}{j + dN^2}, \, \frac{1}{2}\right] \text{ and } \tilde{Z}_j \neq 0.$$ Otherwise the distance remains unchanged, or might even decrease (when $\tilde{Z}_j = 0$). Hence:
    
    \begin{equation}\label{enq_2}
        \tilde{S}_k -  2\sum_{j=1}^k\mathbbm{1}_{\left\{\frac{1}{2} + \frac{(2p-1)\tilde{Z}_{j +dN^{2}}}{j + dN^2}\le  U_{j + dN^{2}} < \frac{1}{2}\right\}} \le \tilde{Z}_{k + dN^{2}} , \, \text{ for all } k\ge 0.
    \end{equation}
    Let $c > 12 + d$. For any $x  < N$ we have
    \begin{align*}\label{eq:boundonethird+}
        \P_x \left(\tau_N > cN^2\right) &= \P_x\left(\tau_N > cN^{2}, \sigma_{2N} > 3(2N)^2\right) + \P_x \left(\tau_N > cN^{2}, \sigma_{2N} \le 3(2N)^2)\right)\\
        &\le \P_x \left(\sigma_{2N} > 3(2N)^2\right) + \P_x \left(\tau_N > cN^{2}, \sigma_{2N} \le 3(2N)^2\right)
    \end{align*}
    But $\P_x \left(\sigma_{2N} > 3(2N)^2 \right) \leq 1/3$ by Markov inequality (see inequality $(\ref{eq:markov_sym})$ in the proof of Lemma \ref{lemma:sym}), hence:
    \begin{equation}\label{eq:boundonethird+}
     \P_x \left(\tau_N > cN^2 \right) \le \frac{1}{3} +  \P_x \left(\tau_N > cN^{2}, \, \sigma_{2N} \le 3(2N)^2 \right).
    \end{equation}
    Now observe that $\tau_N > cN^{2}$ means that the elephant has not reached $N$ yet at time $c N^2$ while $\sigma_{2N} \le 3(2N)^2$ means that the symmetric random walk reached $2N$ in the meantime, even though it was exactly at the same point  $12N^2$ time-steps earlier (or more). Now assume that both hold. Then $\tilde{Z}_k < N$ for any $k \le dN^2 + \sigma_{2N}$, where $dN^2 + \sigma_{2N}$ represents the time at which $(\tilde{S}_k)_{k \geq 0}$ reaches $2N$, but on the $(\tilde{Z}_k)_{k \geq 0}$ time-scale. Then the inequality $\tilde{Z}_{dN^2 + \sigma_{2N}}< N$ together with \eqref{enq_2} implies that 
    $$2N - 2 \sum_{j=1}^{\sigma_{2N}}\mathbbm{1}_{\big\{\frac{1}{2} + \frac{(2p-1)\tilde{Z}_{j + dN^{2}}}{j + dN^{2}}\le  U_{j + dN^{2}} < \frac{1}{2}\big\}} < N.$$
    Moreover, since $\sigma_{2N} \le 3(2N)^2$ and $\tilde{Z}_{j + dN^2} < N$ for any of the indices $j$ in the sum above, the above inequality implies that
    $$ 2N - 2\sum_{j=1}^{3 (2 N)^2}\mathbbm{1}_{\{\frac{1}{2} + \frac{(2p - 1)}{dN}\le U_{j + dN^{2}} < \frac{1}{2} \} } < N.$$

    In short, we just have just proven:
    
    \begin{equation*}
        \big\{ \tau_N > cN^{2}, \sigma_{2N} \le 3(2N)^2 \big\} \subset \Big\{B_N > \frac{N}{2} \Big\},
    \end{equation*}
    where: $$B_N = \sum_{j=1}^{12 N^2}\mathbbm{1}_{\left\{\frac{1}{2} + \frac{(2p - 1)}{dN}\le U_{j + dN^{2}} < \frac{1}{2} \right\}} \sim \text{Binomial}\left(12 N^2, \, \frac{1-2p}{dN}\right).$$ 
    Now, by letting $d > \lceil 24(1-2p)\rceil$, we have:
    \begin{equation}\label{eq_2}
        \mathbb{E}[B_N] = \frac{12(1 - 2p)N}{d} < \frac{N}{2},
    \end{equation}
    and thus
    \begin{equation*}
    	\P\Big(B_N > \frac{N}{2}\Big) \le \frac{1}{2}.
    \end{equation*}
    Hence, for our choice of $d$, it holds that:
    \begin{equation} \label{eq:onehalfpart}
        \P_x \left(\tau_N > cN^{2}, \sigma_{2N} \le 3(2N)^2 \right) \le \frac{1}{2}.
    \end{equation}
    Finally, inequality \eqref{eq:onehalfpart} together with inequality \eqref{eq:boundonethird+} leads to:
    \begin{equation*}
        \P_x \left(\tau_N > cN^2 \right) \le \frac{5}{6}.
    \end{equation*}
    Then the same line of reasoning as in the end of the proof of Lemma \ref{lemma:sym} allows us to conclude:

    \begin{equation*}
        \P_x(\tau_N > t ) \le \frac{6}{5} e^{- \frac{\ln(5/6)}{c} \frac{t}{N^2} }.
    \end{equation*}
    \end{enumerate}

\subsection{Proof of the lower bound in Theorem \ref{main_thm}}

    As for the upper bound, we break the proof into two separate cases. First we get rid of the easy case $p \le 1/2$, which follows as usual from a straightforward comparison between the ERW and the SRW and Lemma \ref{lemma:sym}. Then we address the harder case $ 1/2 < p < 3/4$. The proof strategy is somewhat analogous to that of the upper bound, hinging on controlling the distance between the ERW and the SRW on a quadratic time scale. Specifically, we identify a scaling constant $A$ such that, with positive probability, whenever the ERW reaches $\{-AN, AN\}$, the SRW must have already reached $\{-N, N\}$. For some constant $c$, Lemma~\ref{lemma:sym} then provides a uniform lower bound for $\mathbb{P}(\tau_{AN} > cN^2)$. Moreover, we establish a uniform lower bound for the probability that, after the ERW reaches $\{-AN, AN\}$, it returns to the origin before hitting $\{-2AN, 2AN\}$. This is achieved by constructing a coupling similar to those previously used, but with the SRW replaced by an asymmetric random walk. The conclusion follows readily from these results.

 \medskip
 
 In this subsection let $(\tilde{S}_k)_{k\ge 0}$ be defined as in equation (\ref{srw-coupling}), with $\tilde{S}_0 = 0$ as usual.
 
    \begin{enumerate}
        \item [(i)] \textbf{Case $0 < p \le 1/2$.} In this case the coupling implies: $$\P \left(\sigma_N > t \right) \le \P \left(\tau_N > t \right),$$ and therefore the conclusion simply follows from Lemma \ref{lemma:sym}.
        
        \item [(ii)] \textbf{Case $1/2 < p < 3/4$.} Let us write 
$$
    D_k = 2\sum_{j=1}^k\mathbbm{1}_{\left\{U_{j+1}\in \left(\frac{1}{2}, \frac{1}{2} + \frac{(2p-1)\tilde{Z}_j}{j}\right)\right\}}.
$$
Then 
\begin{align*}
    \mathbb{E}[D_k] &= 2\sum_{j=1}^k\P\left(U_{j+1}\in \left(\frac{1}{2}, \frac{1}{2} + \frac{(2p-1)\tilde{Z}_j}{j} \right) \right) \\
    &= 2(2p - 1)\sum_{j=1}^k\frac{\mathbb{E}[\tilde{Z}_j]}{j}\\
     &= 2(2p - 1)\sum_{j=1}^k\frac{\mathbb{E}[|X_j|]}{j}.
\end{align*}
Since $p < 3/4$, by Jensen's inequality and equation (15) in \cite{schutz2004elephants}, we get: 
$$
    \mathbb{E}[|X_j|]\le \sqrt{\mathbb{E}\big[X_j^2\big]}\le  \frac{\sqrt{j}}{\sqrt{3 - 4p}}.
$$
Then 
$$
    \mathbb{E}[D_k] \le \frac{2(2p - 1)}{\sqrt{3 - 4p}}\sum_{j=1}^k\frac{1}{\sqrt{j}} \le \frac{2(2p - 1)}{\sqrt{3 - 4p}}\left(1 + \int_2^k\frac{1}{\sqrt{x}}dx\right) \le  \frac{4(2p - 1)}{\sqrt{3 - 4p}}\sqrt{k}.
$$
Since 
$$
    \tilde{Z}_k \le \tilde{S}_k + D_k,
$$

For any integer $A \geq 2$, one has:
\begin{align*}
    \P_0(\tau_{AN} > cN^2) &\ge \P_0(\sigma_N > cN^2, D_{cN^2} < (A-1)N)\\
    &= \P_0(\sigma_N > cN^2) - \P_0(\sigma_N > cN^2, D_{cN^2}\ge (A-1)N)\\
    &\ge \P_0(\sigma_N > cN^2) - \P_0(D_{cN^2}\ge (A-1)N)\\
    &\ge \P_0(\sigma_N > cN^2) - \frac{\mathbb{E}[D_{cN^2}]}{(A-1)N}\\
    &\ge \delta - \frac{4(2p - 1)}{\sqrt{3 - 4p}(A-1)}\sqrt{c} = \delta_A.
\end{align*}

Now choose $A$ so that $\delta_A > 0$.

Notice that if $k \ge cN^2$ and $\tilde{Z}_k \le 2AN$ then 
$$
    \frac{(2p-1)\tilde{Z}_k}{k}\le \frac{(2p-1)2AN}{k} \le \frac{2(2p - 1)A}{cN}.
$$
Consider $N$ such that $\frac{2(2p-1)A}{cN} < 1/2$. Let $(Y_k)_{k\ge 0}$ be an asymmetric simple random walk on $\mathbbm{Z}$ that jump upward with probability $p_N = 1/2 + \frac{2(2p-1)A}{cN}$ and let
\begin{equation*}
    T_y = \inf\{k\ge 0: Y_k = y\}.
\end{equation*}
Then we have:
$$
    \mathbbm{P}_{AN}(T_0 < T_{2AN}) = \frac{\varphi(2AN) - \varphi(AN)}{\varphi(2AN) - \varphi(0)},
$$
where $\phi(y) = \left(\frac{1-p_N}{p_N} \right)^y$. This follows from the Gambler ruin problem for the asymmetric simple random walk (see Theorem 4.8.9 in \cite{durrett2019probability}). Observe that 
\begin{equation} \label{eq:assymgambler}
    \frac{\varphi(2AN) - \varphi(AN)}{\varphi(2AN) - \varphi(0)} = \frac{\left(\frac{1-p_N}{p_N}\right)^{2AN} - \left(\frac{1-p_N}{p_N}\right)^{AN} }{\left(\frac{1-p_N}{p_N}\right)^{2AN}  -1} = \frac{\left(\frac{1-p_N}{p_N}\right)^{AN}}{\left(\frac{1-p_N}{p_N}\right)^{AN} + 1},
\end{equation}
and 
\begin{equation*}
    \left(\frac{1-p_N}{p_N}\right) = \frac{cN - 4A(2p-1)}{cN + 4A(2p-2)} = 1 - \frac{8A(2p-1)}{cN+4A(2p-1)}.
\end{equation*}

Since $\log(1 - x) \ge -x/(1-x)$, for $x\in (0,1)$, we have that 
\begin{align*}
    \left(\frac{1-p_N}{p_N}\right)^{AN} &= e^{AN\log\left(1 - \frac{8A(2p-1)}{cN+4A(2p-1)}\right)} \ge e^{-\frac{8A^2(2p-1)N}{cN-4A(2p-1)}}.
\end{align*}
In view of $N > \frac{4(2p-1)A}{c}$, let us take $\epsilon\in (0,c)$ such that $\left\lceil \frac{4(2p-1)A}{c}\right\rceil = \frac{4(2p-1)A}{c - \epsilon}$. Then
\begin{equation*}
    cN - 4A(2p-1) \ge \epsilon N,
\end{equation*}


which implies that 
\begin{equation} \label{eq:belowkappaabove1}
    1 \ge \left(\frac{1-p_N}{p_N}\right)^{AN} \ge e^{-\frac{8A^2(2p-1)N}{cN-4A(2p-1)}}\ge e^{-\frac{8A^2(2p-1)}{\epsilon}} \eqcolon \kappa_A.
\end{equation}
Then, from \eqref{eq:assymgambler} and \eqref{eq:belowkappaabove1}, we get: 
\begin{equation*}
    \mathbbm{P}_{AN}(T_0 < T_{2AN}) \ge \frac{\kappa_A}{2} \eqcolon \gamma_A > 0.
\end{equation*}



Now, if $(\tilde{Z}_n)_{n \geq 0}$ hits $AN$ after $cN^2$ steps, coupling $(\tilde{Z}_n)_{n \geq 0}$ with $(Y_n)_{n \geq 0}$ in the usual way, we obtain from the inequality above that $\tilde{Z}_k$ returns to zero before hitting $2AN$ with a probability bigger than $\gamma_A$. Then, the same arguments as before lead to:
$$
    \P(\tau_{2AN} \ge kcN^2) \ge (\delta_A\gamma_A)^k, \, \text{ for all } k\ge 1 \, \text{ and } \, N > \frac{4(2p-1)A}{c}.
$$
Let us denote $ M= \left\lfloor N/2A \right\rfloor $ and $\delta = \delta_A \gamma_A \in (0,1)$. Then if $M > \frac{4(2p-1)A}{c}$, we have for any $t \in \N$:
$$ \P \left( \tau_N > t \right)  \geq \P \left( \tau_{2AM} > \frac{t}{cM^2} cM^2\right) \geq \P \left( \tau_{2AM} > \left\lceil \frac{t}{cM^2} \right\rceil cM^2\right) \geq \delta^{\left\lceil \frac{t}{cM^2} \right\rceil} \geq \delta^{\frac{t}{cM^2} + 1}$$

Now, by definition of $M$, one has $M \geq \frac{N - 2A}{2A}$ and therefore, assuming $N \geq 2A + 1$, one has also $M^2 \geq \left(\frac{N - 2A}{2A}\right)^2$. Hence, for $N \geq 2A + 1$, one has $$\frac{t}{cM^2} \leq \frac{4A^2 t}{c(N-2A)^2} = \left(\frac{N}{N-2A} \right)^2 \frac{4A^2 t}{c N^2},$$ and since $\frac{N}{N-2A} \leq 2A+1$, for $N \geq 2A + 1$ one has also $$\frac{t}{cM^2} \leq \frac{4(2A+1)^2A^2}{c} \frac{t}{N^2}.$$ 

Then, for $N > \max \left\{2A, \frac{8A^2(2p-1)}{c} \right\}$, we indeed obtain $$\P \left( \tau_N > t \right) \geq \delta^{\frac{4(2A+1)^2A^2}{cN^2} t + 1} = C_1 e^{-\tilde{C}_2 \frac{t}{N^2}},$$
with $C_1 = \delta$ and $\tilde{C}_2 = \frac{4(2A+1)^2A^2}{c} \log(\frac{1}{\delta})$.

Now assume that $2\le N\le \max \left\{2A, \frac{8A^2(2p-1)}{c} \right\}:= C_A$. Let $\hat{S}_0 = 0$ and for $k \geq 1$ define 
\begin{equation*}
    \hat{S}_k \coloneq \sum_{j=1}^k \mathbbm{1}_{\left\{U_j < p \right\}}.
\end{equation*}

Notice that, since ${\left\{U_j < p \right\} = \left\{U_j < \frac{1}{2} + \frac{2p-1}{2} \right\}}$, one has $\tilde{Z}_k \le \hat{S}_k$, for all $k\ge 0$. Therefore, if 
\begin{equation*}
    \hat{\sigma}_N = \inf\{k\ge 1: \hat{S}_k = N\},
\end{equation*}
then, for all $N \ge 2$ and $t\ge 1$, one has: 
\begin{equation*}
    \mathbb{P}(\tau_N > t) \ge \mathbb{P}(\hat{\sigma}_N > t).
\end{equation*}
Observing that $\{U_1 > p , \, \ldots , \, U_t > p\}\subset \{\hat{\sigma}_N > t\}$, we obtain the following---very crude, but good  enough---bound:
\begin{equation*}
    \mathbb{P}(\hat{\sigma}_N > t) \ge (1-p)^t = e^{\log(1 - p) t} \ge e^{-\log \left(\frac{1}{1-p}\right)C_A^2 \frac{t}{N^2}},
\end{equation*}
for all $t\ge 1$ and $2\le N\le C_A$. Hence taking $C_2 = \max \left\{\log \left(\frac{1}{1-p}\right)C_A^2, \tilde{C}_2 \right \}$, we can indeed conclude:
\begin{equation*}
    \mathbbm{P}(\tau_N \ge t) \ge C_1e^{-C_2\frac{t}{N^2}},
\end{equation*}
for all $N\ge 2$ and $t\ge 1$.
\end{enumerate}

\section{Expected duration of Elephant's Ruin}

Finally, we prove Proposition \ref{prop_1}. The two main ingredients of proof are the functional limit obtained in \cite{baur2016elephant}, together with the exponential bounds on the tail proven in the precedent sections.

\medskip

\begin{proof}[Proof of Proposition \ref{prop_1}]

We assume $p\in (0,3/4)$. It was established in \cite{baur2016elephant} that:
\begin{equation*}
    \left(\frac{Z_{\lfloor nt\rfloor}}{\sqrt{n}}\right)_{t\ge 0}\underset{N \rightarrow \infty}{\longrightarrow}  (W_t)_{t\ge 0},
\end{equation*}
where the convergence is a convergence in distribution in the  Skorokhod space $D[0,\infty)$ of right-continuous functions with left-hand limits, and $(W_t, t\ge 0)$ is a continuous $\mathbbm{R}$-valued Gaussian process satisfying  $W_0 = 0$, $\mathbbm{E}[W_t] = 0$ for all $t\ge 0$ and 
\begin{equation*}
    \mathbbm{E}[W_sW_t] = \frac{s}{3 - 4p}\left(\frac{t}{s} \right)^{2p-1}, \, \text{ for } 0 < s \le t.
\end{equation*} 
Let $\nu$ be the stopping time defined as: 
\begin{equation*}
    \nu = \inf\{t\ge 0: |W_t| = 1\}.
\end{equation*}
Since 
\begin{equation*}
    \left\{ \frac{\tau_N}{N^2} >  t\right\} = \left\{\tau_N > tN^2 \right\} = \left\{ \sup_{0\le s\le t}|Z_{\lfloor sN^2\rfloor}| < N\right\} = \left\{ \sup_{0\le s\le t}\frac{|Z_{\lfloor sN^2\rfloor}|}{N} < 1 \right\},
\end{equation*}
the Continuous Mapping Theorem\footnote{See, for example, Theorem 5.1 in \cite{bill}.} implies the following convergence in distribution:  
\begin{equation*}
    \frac{\tau_N}{N^2} \underset{N \rightarrow \infty}{\longrightarrow} \nu.
\end{equation*}
In order to derive convergence of the means from the the convergence in distribution,  and in order to establish $\mathbbm{E}[\nu] < +\infty$, it is enough to prove that the sequence $\left(\tau_N/N^2\right)_{N\ge 1}$ is uniformly integrable --- see, for example, Theorem 25.12 in \cite{billingsley}. That is, we need to show that:

$$ \lim_{M \rightarrow \infty} \left( \sup_{N \in \N^*} \E \left[ \frac{\tau_N}{N^2} \mathbbm{1}_{\left\{\tau_N > M N^2\right\}} \right]\right) = 0.$$
This follows easily from the upper bound obtained earlier. Indeed:
\begin{align*}
    \frac{\tau_N}{N^2} \mathbbm{1}_{\{\tau_N > MN^2\}} &= \sum_{k=M}^\infty \frac{\tau_N}{N^2} \mathbbm{1}_{\{kN^2 < \tau_N \le (k+1)N^2\}}\\
    &\le \sum_{k=M}^\infty (k+1) \mathbbm{1}_{\{kN^2 < \tau_N \le (k+1)N^2\}}\\
    &\le \sum_{k=M}^\infty (k+1) \mathbbm{1}_{\{\tau_N > k N^2\}}
\end{align*}
By Tonelli's Theorem, and Theorem \ref{main_thm}: $$ \sup_{N \in \N^*} \E \left[\frac{\tau_N}{N^2} \mathbbm{1}_{\{\tau_N > MN^2\}} \right] \leq \sum_{k=M}^\infty (k+1) \sup_{N \in \N^*} \P \left(\tau_N > k N^2 \right) \le C_1 \sum_{k=M}^\infty (k+1) e^{-C_2k},$$ and since $\sum_{k=1}^\infty (k+1) e^{-C_2 k} < \infty$ we indeed have $\sum_{k=M}^\infty (k+1) e^{-C_2k} \underset{M \rightarrow \infty}{\longrightarrow} 0$. Therefore:
$$ 
\frac{\E \left[ \tau_N \right]}{N^2} \underset{N \rightarrow \infty}{\longrightarrow} \E[\nu].
$$
This gives us the desired result  taking $\theta = \mathbbm{E}[\nu]$.

\end{proof}

\appendix

\section{Proof of Lemma \ref{lemma:sym}}\label{appendix-A}

For the sake of completeness, we prove here the exponential bounds of Lemma 1, for which were unable to find a relevant reference even though they must be common knowledge.

\begin{proof}[Proof of Lemma \ref{lemma:sym}]

We start with the upper bound and then proceed to the lower bound. Fix some $c > 1$. For any $x \in \llbracket -N, N \rrbracket$, a classical martingale argument (see for example Theorem 4.8.7  in \cite{durrett2019probability}) gives us:
    \begin{equation*}
        \E_x \left[\sigma_N \right] = (N-x)(x+N) \le N^2.
    \end{equation*}
    Hence, by Markov inequality:
    \begin{equation}\label{eq:markov_sym}
        \P_x \left(\sigma_N > cN^2 \right) \le \frac{\E_x \left[ \sigma_N \right]}{cN^2}\le \frac{1}{c}.
    \end{equation}
    For $t \geq 0$, it follows that:
    \begin{equation*}
        \P(\sigma_N > t) \leq \P \left(\sigma_N \ge \left\lfloor \frac{t}{cN^2}\right\rfloor cN^2 \right) \le \left(\frac{1}{c}\right)^{\left\lfloor \frac{t}{cN^2}\right\rfloor} \le \left(\frac{1}{c}\right)^{\frac{t}{cN^2} - 1} = C_3e^{-C_4\frac{t}{N^2}},
    \end{equation*}
    with $C_3 = c$ and $C_4 = \log(c)/c$. Notice that, for $t \geq cN^2$, the second inequality above is a consequence of \eqref{eq:markov_sym} together with the strong Markov property, whereas it is trivial for $t < c N^2$, since in that case $\left(\frac{1}{c}\right)^{\left\lfloor \frac{t}{cN^2} \right\rfloor} = 1$.

    Now, to obtain the lower bound, fix another $c\in (0,1)$. By Kolmogorov's maximal inequality\footnote{See, for example, Theorem 2.5.5 in \cite{durrett2019probability}.} we have:
    \begin{equation*}
        \P\left(\max_{1\le i \le \left\lfloor cN^2 \right\rfloor} |S_i|\ge N\right) \le \frac{\Var \left(S_{\left\lfloor cN^2\right\rfloor} \right)}{N^2} \le c.
    \end{equation*}
    Then
    \begin{equation*}
        \P \left(\sigma_N > cN^2 \right) = \P\left(\max_{1\le i \le \left\lfloor cN^2 \right\rfloor} |S_i|< N\right)= 1 - \P\left(\max_{1\le i \le \left\lfloor cN^2 \right\rfloor} |S_i|\ge N\right)\ge 1 - c \eqcolon \delta.
    \end{equation*}
    Now let $\nu_0 = 0$ and for $i \geq 1$:
    $$
        \tau_i = \inf\{n\ge \nu_{i-1}: |S_n| = N\} \; \; \; \text{ and } \; \; \; \nu_i = \inf\{n\ge \tau_i: S_n\in \{0,2N\}\}.
    $$
    Define the following events:
    $$
        E_i = \{\tau_i - \nu_{i-1} > cN^2 \}\, \text{ and } \, F_i = \{S_{\nu_i} = 0 \},
    $$
    and observe that
    $$
        \bigcap_{i=1}^k(E_i\cap F_i)\subset \{\sigma_{2N} > kcN^2 \}.
    $$
    Since
    $$
    \P\left(E_1\cap F_1\right) = \P_0 \left(\sigma_{N} > cN^2\right) \P_N \left(\sigma_0 < \sigma_{2N} \right) \geq \frac{\delta}{2},
    $$
    the strong Markov's property implies:
    \begin{equation*}
        \P(\sigma_{2N} > kcN^2) \ge  \left(\P\left(E_1\cap F_1\right)\right)^k \ge \left(\frac{\delta}{2}\right)^k.
    \end{equation*}
    
    From there it is not hard obtain some constants $C_1$ and $C_2$ such that 
    $$
        \mathbb{P}(\sigma_N  > t)\ge C_1e^{-C_2\frac{t}{N^2}}, \, \text{ for all } N\ge 2 \text{ and } t\ge 1.
    $$ 
\end{proof}

\noindent{\bf Acknowledgements.} We are deeply grateful to Leonardo T. Rolla for a particularly fruitful discussion during the development of this work, as well as for his insightful comments and valuable suggestions on the final manuscript.
 
\bibliographystyle{acm}

\bibliography{reference}	

\end{document}